\documentclass[11pt]{article}
\usepackage{amsmath, amsthm, amssymb, footmisc, graphicx, tabularx, hyperref, bbm, nicefrac}
\usepackage[utf8]{inputenc}

\newcommand{\be}{\begin{equation}}
\newcommand{\ee}{\end{equation}}
\newcommand{\bea}{\begin{eqnarray}}
\newcommand{\eea}{\end{eqnarray}}
\newcommand{\beas}{\begin{eqnarray*}}
\newcommand{\eeas}{\end{eqnarray*}}

\newcommand{\crl}[1]{\ensuremath{ \left\{ #1 \right\} }}
\newcommand{\edg}[1]{\ensuremath{\! \left[ #1 \right] }}
\newcommand{\brak}[1]{\ensuremath{\left( #1 \right)}}

\newtheorem{theorem}{Theorem}%[section]
\newtheorem{proposition}[theorem]{Proposition}

\newtheorem{corollary}[theorem]{Corollary}
\newtheorem{remark}[theorem]{Remark}
\newenvironment{Remark}{\begin{remark}\rm}{\end{remark}}

\title{Deep optimal stopping\footnote{We thank Philippe Ehlers, Ariel Neufeld and Martin Stefanik for 
fruitful discussions and helpful comments.}}
\author{Sebastian Becker\footnote{Zenai AG, 8045 Zurich, Switzerland; sebastian.becker@zenai.ch}, \;
Patrick Cheridito\footnote{RiskLab, Department of Mathematics, ETH Zurich, 8092 Zurich, Switzerland; 
patrick.cheridito@math.ethz.ch} \; \& \; Arnulf Jentzen\footnote{SAM, Department of Mathematics, ETH Zurich, 
8092 Zurich, Switzerland; arnulf.jentzen@sam.math.ethz.ch}}

\date{}
\begin{document}

\maketitle

\begin{abstract}
In this paper we develop a deep learning method for optimal stopping problems which
directly learns the optimal stopping rule from Monte Carlo samples. As such, it is broadly applicable 
in situations where the underlying randomness can efficiently be simulated. We test the 
approach on three problems: the pricing of a Bermudan max-call option, the pricing of a callable 
multi barrier reverse convertible and the problem of optimally stopping a fractional Brownian motion. 
In all three cases it produces very accurate results in high-dimensional situations with short 
computing times.\\[2mm]
\end{abstract} 

\section{Introduction}
\label{sec:intro}

We consider optimal stopping problems of the form $ \sup_{ \tau } \mathbb{E} \, g( \tau, X_{ \tau })$,
where $X= (X_n)_{n=0}^N$ is an $\mathbb{R}^d$-valued discrete-time Markov process 
and the supremum is over all stopping times $\tau$ based on observations of $X$. 
Formally, this just covers situations where the stopping decision can only be made at finitely many times.
But practically all relevant continuous-time stopping problems can be approximated with time-discretized versions.
The Markov assumption means no loss of generality. We make it because it simplifies the presentation and many 
important problems already are in Markovian form. But every optimal stopping problem can 
be made Markov by including all relevant information from the past in the current state of $X$
(albeit at the cost of increasing the dimension of the problem).

In theory, optimal stopping problems with finitely many stopping opportunities can be solved exactly.
The optimal value is given by the smallest supermartingale that dominates the reward process -- the 
so-called Snell envelope -- and the smallest (largest) optimal stopping time is the first time the
immediate reward dominates (exceeds) the continuation value; see, e.g., \cite{PS06, LL08}. 
However, traditional numerical methods suffer from the curse of dimensionality. For instance, the 
complexity of standard tree- or lattice-based methods increases exponentially in the dimension. 
For typical problems they yield good results for up to three dimensions. 
To treat higher-dimensional problems, various Monte Carlo based methods have been developed over the last 
years. A common approach consists in estimating continuation values to either derive stopping rules or 
recursively approximate the Snell envelope; see e.g.,
\cite{Till93, BM95, Ca96, LS01, TVR01, BKT03, BG04, BPP05, KS06, EKT07, BS08, JO15, BSST} or
\cite{HK04, KKT10}, which use neural networks with one hidden layer to do this.
A different strand of the literature has focused on approximating optimal exercise boundaries; 
see, e.g., \cite{A00, G03, Be11AA}. Based on an idea of \cite{DK94}, a dual approach was
developed by \cite{R02, HK04}; see \cite{J07, CG07} for a multiplicative version and 
\cite{AB04, BC08, BBS09, R10, DFM12, Be13, BSD13, Le16} for extensions and primal-dual methods.
In \cite{SS18} optimal stopping problems in continuous time are treated by approximating the 
solutions of the corresponding free boundary PDEs with deep neural networks. 

In this paper we use deep learning to approximate an optimal stopping time. Our approach is related to policy
optimization methods used in reinforcement learning \cite{SB}, deep reinforcement learning \cite{SLM, MKS, SHM, LHP}
and the deep learning method for stochastic control problems proposed by \cite{HE}.
However, optimal stopping differs from the typical control problems studied in this literature.
The challenge of our approach lies in the implementation of a deep learning method that can efficiently 
learn optimal stopping times. We do this by decomposing an optimal stopping time into a sequence of 0-1 
stopping decisions and approximating them recursively with a sequence of multilayer feedforward neural 
networks. We show that our neural network policies can approximate optimal stopping times to any degree of 
desired accuracy. A candidate optimal stopping time $\hat{\tau}$ can be obtained by 
running a stochastic gradient ascent. The corresponding expectation 
$\mathbb{E} \, g(\hat{\tau},X_{\hat{\tau}})$ provides a lower bound 
for the optimal value $\sup_{\tau} \mathbb{E} \, g(\tau,X_{\tau})$. 
Using a version of the dual method of \cite{R02,HK04}, we also derive an upper bound.
In all our examples, both bounds can be computed with short run times and lie close together.

The rest of the paper is organized as follows: In Section \ref{sec:dl} we introduce the 
setup and explain our method of approximating optimal stopping times with neural networks. 
In Section \ref{sec:bci} we construct lower bounds, upper bounds, point estimates and confidence intervals
for the optimal value. In Section \ref{sec:ex} we test the approach on three examples: the pricing of a Bermudan 
max-call option on different underlying assets, the pricing of a callable multi barrier reverse convertible 
and the problem of optimally stopping a fractional Brownian motion. In the first two examples, we use a
multi-dimensional Black--Scholes model to describe the dynamics of the underlying assets. Then the pricing of 
a Bermudan max-call option amounts to solving a $d$-dimensional optimal stopping problem,
where $d$ is the number of assets. We provide numerical results for $d = 2,3,5, 10, 20, 30, 50, 100, 200$ and 500.
In the case of a callable MBRC, it becomes a $d+1$-dimensional stopping problem since
one also needs to keep track of the barrier event. We present results for 
$d = 2,3,5,10,15$ and $30$. In the third example we only consider a one-dimensional fractional Brownian 
motion. But fractional Brownian motion is not Markov. In fact, all of its increments are correlated. So, 
to optimally stop it, one has to keep track of all past movements. To make it tractable, we approximate
the continuous-time problem with a time-discretized version, which if formulated as a Markovian problem,
has as many dimensions as there are time-steps. We compute a solution for 100 time-steps.

\section{Deep learning optimal stopping rules}
\label{sec:dl}

Let $X = (X_n)_{n=0}^N$ be an $\mathbb{R}^d$-valued discrete-time Markov process
on a probability space $(\Omega, {\cal F} , \mathbb{P})$, where $N$ and $d$ are 
positive integers. We denote by ${\cal F}_n$ the $\sigma$-algebra generated by 
$X_0, X_1, \dots, X_n$ and call a random variable $\tau \colon \Omega \to \crl{0,1, \dots, N}$
an $X$-stopping time if the event $\crl{\tau = n}$ belongs to $ {\cal F}_n $ for all 
$n \in \{ 0, 1, \dots, N \}$.

Our aim is to develop a deep learning method that can efficiently learn an 
optimal policy for stopping problems of the form
\be 
\label{os}
\sup_{ \tau \in {\cal T} } \mathbb{E}\, g(\tau,X_{\tau}),
\ee
where $g \colon \crl{0,1,\dots, N} \times \mathbb{R}^d \to \mathbb{R}$ is a measurable function
and ${\cal T}$ denotes the set of all $X$-stopping times. To make sure that problem \eqref{os} 
is well-defined and admits an optimal solution, we assume that $g$ satisfies the integrability condition
\be \label{ic}
\mathbb{E}\, | g(n,X_n) | < \infty \quad \mbox{for all } n \in \crl{0,1,\dots,N};
\ee
see, e.g., \cite{PS06, LL08}. To be able to derive confidence intervals for the optimal value \eqref{os},
we will have to make the slightly stronger assumption 
\be \label{ic2}
\mathbb{E} \edg{ g(n,X_n)^2} < \infty \quad \mbox{for all } n \in \crl{0,1,\dots,N}
\ee
in Subsection \ref{subsec:ci} below. This is satisfied in all our examples in Section \ref{sec:ex}.

\subsection{Expressing stopping times in terms of stopping decisions} 

Any $X$-stopping time can be decomposed into a sequence of 0-1 stopping decisions.
In principle, the decision whether to stop the process at time $n$ if it has not 
been stopped before, can be made based on the whole evolution of $X$ from time $0$ until $n$. 
But to optimally stop the Markov process $X$, it is enough to make stopping decisions according to 
$f_n(X_n)$ for measurable functions $f_n \colon \mathbb{R}^d \to \crl{0,1}$, $ n = 0, 1, \dots, N$.
Theorem~\ref{thm:rep} below extends this well-known fact and serves as the theoretical basis of our method.

Consider the auxiliary stopping problems 
\be \label{nos}
V_n = \sup_{\tau \in {\cal T}_n} \mathbb{E} \, g(\tau,X_{\tau})
\ee
for $n = 0, 1, \dots, N$, where ${\cal T}_n$ is the set of all $X$-stopping times satisfying $n \le \tau \le N$. 
Obviously, ${\cal T}_N$ consists of the unique element $\tau_N \equiv N$,
and one can write $\tau_N = N f_N(X_N)$ for the constant function $f_N \equiv 1$. 
Moreover, for given $n \in \crl{0, 1, \dots, N} $ and a sequence of measurable functions
$ f_n, f_{ n + 1 }, \dots, f_N \colon \mathbb{R}^d \to \crl{0,1}$ with $f_N \equiv 1$, 
\be \label{taun}
\tau_n =  \sum_{m = n}^{N} m f_m(X_m) \prod_{j=n}^{m - 1} \left( 1 - f_j(X_j) \right)
\ee
defines\footnote{In expressions of the form \eqref{taun}, we understand the empty product 
$\prod_{j=n}^{n-1} \left( 1 - f_j(X_j) \right)$ as $1$.} a stopping time in ${\cal T}_n$.
The following result shows that, for our method of recursively computing 
an approximate solution to the optimal stopping problem \eqref{os}, it will be sufficient to consider stopping 
times of the form \eqref{taun}.

\begin{theorem} \label{thm:rep}
For a given $n \in \crl{0,1,\dots, N-1}$, let $\tau_{n+1}$ be a stopping time in ${\cal T}_{n+1}$ of the form
\be \label{taun1}
\tau_{n+1} = \sum_{m =n+1}^N mf_m(X_m) \prod_{j=n+1}^{m-1} (1-f_j(X_j))
\ee
for measurable functions $f_{n+1}, \dots, f_N \colon \mathbb{R}^d \to \crl{0,1}$ with $f_N \equiv 1$.
Then there exists a measurable function $f_n \colon \mathbb{R}^d \to \crl{0,1}$ such that 
the stopping time $\tau_n \in {\cal T}_n$ given by \eqref{taun} satisfies
\[
\mathbb{E} \, g( \tau_n, X_{ \tau_n }) \ge V_n -  \left(V_{ n + 1 } - \mathbb{E} \, g( \tau_{n+1}, X_{ \tau_{n+1}}) \right),
\]
where $V_n$ and $V_{n+1}$ are the optimal values defined in \eqref{nos}.
\end{theorem}

\begin{proof}
Denote $\varepsilon = V_{ n + 1 } - \mathbb{E} \, g( \tau_{ n + 1 } , X_{ \tau_{ n + 1 } })$, 
and consider a stopping time $ \tau \in {\cal T}_n$. By the Doob--Dynkin lemma (see, e.g., Theorem 4.41 in \cite{AB}),
there exists a measurable function $h_n \colon \mathbb{R}^d \to \mathbb{R}$ such that 
$h_n(X_n)$ is a version of the conditional expectation $\mathbb{E} \edg{g(\tau_{n+1},X_{\tau_{n+1}}) \mid X_n}$.
Moreover, due to the special form \eqref{taun1} of $\tau_{n+1}$, 
\[
g(\tau_{n+1},X_{\tau_{n+1}}) = \sum_{m =n+1}^N g(m,X_m) 1_{\crl{\tau_{n+1} = m}}
= \sum_{m =n+1}^N g(m,X_m) 1_{\crl{f_m(X_m) \prod_{j=n+1}^{m-1} (1-f_j(X_j))=1 }}
\]
is a measurable function of $X_{n+1}, \dots, X_N$. So it follows from the Markov property of $X$ that 
$h_n(X_n)$ is also a version of the conditional expectation 
$\mathbb{E} \edg{g(\tau_{n+1},X_{\tau_{n+1}}) \mid {\cal F}_n}$. Since the events
\[
D = \crl{ g(n,X_n) \ge h_n(X_n)}\quad \mbox{and} \quad  E = \crl{\tau = n}
\]
are in ${\cal F}_n$, $\tau_n = n 1_D + \tau_{n+1} 1_{D^c }$ belongs to
${\cal T}_n $ and $\tilde{\tau} = \tau_{n+1} 1_E + \tau 1_{E^c}$ to ${\cal T}_{n+1}$.
It follows from the definitions of $ V_{n+1} $ and $ \varepsilon $ that
$\mathbb{E} \, g( \tau_{ n + 1 }, X_{ \tau_{ n + 1 } }) = V_{n+1} - \varepsilon
\ge \mathbb{E} \, g( \tilde{ \tau }, X_{ \tilde{\tau} } )- \varepsilon$. Hence,
\[
\mathbb{E} \edg{ g( \tau_{ n + 1 }, X_{ \tau_{ n + 1 } } ) 1_{E^c}}
\ge \mathbb{E} \edg{g( \tilde{\tau} , X_{ \tilde{ \tau } } ) 1_{E^c}} - \varepsilon = 
\mathbb{E} \edg{g( \tau, X_{ \tau } ) 1_{E^c}} - \varepsilon,
\]
from which one obtains
\[
\begin{split}
& \mathbb{E} \, g( \tau_n, X_{ \tau_n } ) = \mathbb{E} \edg{g(n,X_n ) I_D + g(\tau_{n+1}, X_{\tau_{n+1}}) I_{D^c}}
= \mathbb{E} \edg{g(n,X_n ) I_D + h_n(X_n) I_{D^c}}\\
& \ge \mathbb{E} \edg{g( n, X_n ) I_E + h_n(X_n) I_{E^c}} = 
\mathbb{E} \edg{g( n, X_n ) I_E + g( \tau_{ n + 1 } , X_{ \tau_{ n + 1 } } ) I_{E^c}}\\
& \ge \mathbb{E} \edg{g( n, X_n ) I_E +  g( \tau, X_{ \tau } ) I_{E^c}} - \varepsilon = 
\mathbb{E} \, g( \tau, X_{ \tau }) - \varepsilon .
\end{split}
\]
Since $\tau \in {\cal T}_n $ was arbitrary, this shows that
$\mathbb{E} \, g(\tau_n, X_{ \tau_n } ) \ge V_n - \varepsilon$. Moreover, one has
$1_D = f_n(X_n)$ for the function $f_n \colon \mathbb{R}^d \to \crl{0,1}$ given by
\[
f_n(x) = 
  \begin{cases}
    1 & \mbox{ if } g(n,x) \ge h_n(x)
  \\
    0 & \mbox{ if } g(n,x) < h_n(x)
  \end{cases}.
\]
Therefore,
$$
\tau_n = n f_n(X_n) + \tau_{n+1} (1-f_n(X_n)) = \sum_{m =n}^{N} mf_m(X_m) \prod_{j=n}^{m-1} (1-f_j(X_j)),
$$
which concludes the proof.
\end{proof}

\begin{Remark}
Since for $f_N \equiv 1$, the stopping time $\tau_N = f_N(X_N)$ is optimal in ${\cal T}_N$,
Theorem \ref{thm:rep} inductively yields measurable functions $f_n \colon \mathbb{R}^d \to \crl{0,1}$ such that 
for all $n \in \crl{0,1,\dots, N-1}$, the stopping time $\tau_n$ given by \eqref{taun} is optimal among ${\cal T}_n$.
In particular, 
\be \label{ost}
\tau = \sum_{n = 1}^{N} nf_n(X_n) \prod_{j=0}^{n-1} (1-f_j(X_j)) 
\ee
is an optimal stopping time for problem \eqref{os}.
\end{Remark}

\begin{Remark}
In many applications, the Markov process $X$ starts from a deterministic initial value $x_0 \in \mathbb{R}^d$.
Then the function $f_0$ enters the representation \eqref{ost} only through the value $f_0(x_0)\in \crl{0,1}$;
that is, at time $0$, only a constant and not a whole function has to be learned.
\end{Remark}

\subsection{Neural network approximation}

Our numerical method for problem \eqref{os} consists in iteratively approximating optimal stopping 
decisions 
$
  f_n \colon \mathbb{R}^d \to \crl{0,1} 
$, 
$ 
  n = 0, 1, \dots, N - 1,
$
by a neural network
$f^{ \theta } \colon \mathbb{R}^d \to \crl{ 0, 1 }$
with parameter 
$
  \theta \in \mathbb{R}^q 
$. 
We do this by starting with the terminal stopping decision $ f_N \equiv 1 $ 
and proceeding by backward induction. 
More precisely, let $ n \in \crl{ 0, 1, \dots, N - 1 } $, and assume parameter values 
$
  \theta_{ n + 1 }, \theta_{ n + 2 }, \dots, \theta_N \in \mathbb{R}^q 
$ 
have been found such that $f^{\theta_N} \equiv 1$ and the stopping time
\[
\tau_{n+1} = \sum_{m =n+1}^{N} m f^{\theta_m}(X_m) \prod_{j=n+1}^{m-1} (1-f^{\theta_j}(X_j))
\]
produces an expected value 
$ 
\mathbb{E} \, g(\tau_{n+1},X_{\tau_{n+1}})
$ 
close to the optimum $ V_{ n + 1 } $. Since $f^{\theta}$ takes values in $\crl{0,1}$, it does not directly 
lend itself to a gradient-based optimization method. So, as an intermediate step, we
introduce a feedforward neural network $F^{\theta} \colon \mathbb{R}^d \to (0,1)$ 
of the form 
\[
F^{\theta} = \psi \circ a^{\theta}_I \circ \varphi_{q_{I-1}} \circ a^{\theta}_{I-1} \circ \dots \circ \varphi_{q_1} \circ a^{\theta}_1,
\]
where 
\begin{itemize}
\item 
$I, q_1, q_2, \dots, q_{I-1}$ are positive integers specifying the depth of the network and
the number of nodes in the hidden layers (if there are any),
\item 
$a^{\theta}_1 \colon \mathbb{R}^d \to \mathbb{R}^{q_1}, \dots, 
a^{\theta}_{I-1} \colon \mathbb{R}^{q_{I-2}} \to \mathbb{R}^{q_{I-1}} $ and 
$a^{\theta}_I \colon \mathbb{R}^{q_{I-1}} \to \mathbb{R}$ are affine functions,
\item 
for $j \in \mathbb{N}$, $\varphi_j \colon \mathbb{R}^j \to \mathbb{R}^j$ is the 
component-wise ReLU activation function given by \linebreak $\varphi_j(x_1, \dots, x_j) = (x^+_1 , \dots, x^+_j)$
\item 
$\psi \colon \mathbb{R} \to (0,1)$ is the standard logistic function 
$\psi(x) = e^x/(1+ e^x) = 1 / ( 1 + e^{ - x } )$.
\end{itemize}
The components of the parameter $\theta \in \mathbb{R}^q$ of $F^{\theta}$ consist of the 
entries of the matrices $A_1 \in \mathbb{R}^{q_1 \times d}, \dots,$  $A_{I-1} \in \mathbb{R}^{q_{I-1} \times q_{I-2}},
A_I \in \mathbb{R}^{1 \times q_{I-1}}$ and the vectors $b_1 \in \mathbb{R}^{q_1}, \dots, b_{I-1} \in \mathbb{R}^{q_{I-1}},
b_I \in \mathbb{R}$ given by the representation of the affine functions 
\[
a^{\theta}_i(x) = A_i x + b_i, \quad i = 1, \dots, I.
\]
So the dimension of the parameter space is 
\[
q = 
\begin{cases}
d+1 & \mbox{ if } I =1 \\
1 + q_1 + \dots + q_{I-1} + d q_1 + \dots + q_{I-2} q_{I-1} + q_{I-1} & \mbox{ if } I \ge 2,
\end{cases}
\]
and for given $x \in \mathbb{R}^d$, $F^{\theta}(x)$ is continuous as well as almost everywhere smooth in $\theta$.
Our aim is to determine $\theta_n \in \mathbb{R}^q$ so that 
\[
\mathbb{E} \edg{g(n,X_n) F^{\theta_n}(X_n) + g(\tau_{n+1}, X_{\tau_{n+1}}) (1- F^{\theta_n}(X_n))}
\]
is close to the supremum 
$\sup_{\theta \in \mathbb{R}^q} \mathbb{E} \edg{g(n,X_n) F^{\theta}(X_n) + g(\tau_{n+1}, X_{\tau_{n+1}}) (1- F^{\theta}(X_n))}$.
 Once this has been achieved, we define the function 
$f^{\theta_n} \colon \mathbb{R}^d \to \crl{0,1}$ by 
\be \label{ftheta}
f^{\theta_n} = 1_{[0,\infty)} \circ a^{\theta_n}_I
\circ \varphi_{q_{I-1}} \circ a^{\theta_n}_{I-1} \circ \dots \circ \varphi_{q_1} \circ a^{\theta_n}_1,
\ee
where $1_{[0,\infty)} \colon \mathbb{R} \to \crl{0,1}$ is the indicator function of $[0,\infty)$.
The only difference between $F^{\theta_n}$ and $f^{\theta_n}$ is the final nonlinearity. While
$F^{\theta_n}$ produces a stopping probability in $(0,1)$, the output of $f^{\theta_n}$ is a hard 
stopping decision given by $0$ or $1$, depending on whether $F^{\theta_n}$ takes a value below or above $1/2$.

The following result shows that for any depth $I \ge 2$, a neural network of the form \eqref{ftheta} 
is flexible enough to make almost optimal stopping decisions provided it has sufficiently many nodes.

\begin{proposition} \label{prop:appr}
Let $n \in \crl{0,1,\dots, N-1}$ and fix a stopping time $\tau_{n+1} \in {\cal T}_{n+1}$. Then, for 
every depth $I \ge 2$ and constant $\varepsilon > 0$, there exist positive integers $q_1, \dots, q_{I-1}$ such that 
\beas
&& \sup_{\theta \in \mathbb{R}^q} \mathbb{E} \edg{g(n,X_n) f^{\theta}(X_n) 
+ g(\tau_{n+1}, X_{\tau_{n+1}}) (1- f^{\theta}(X_n))}\\
&&\ge \sup_{f \in {\cal D}} \mathbb{E} \edg{g(n,X_n) f(X_n) + g(\tau_{n+1}, X_{\tau_{n+1}}) (1- f(X_n))} - \varepsilon,
\eeas
where ${\cal D}$ is the set of all measurable functions $f \colon \mathbb{R}^d \to \crl{0,1}$.
\end{proposition}

\begin{proof}
Fix $\varepsilon > 0$. It follows from the integrability condition \eqref{ic} 
that there exists a measurable function $\tilde{f} \colon \mathbb{R}^d \to \crl{0,1}$ such that
\be \label{fhat}
\begin{aligned}
& \mathbb{E} \edg{g(n,X_n) \tilde{f}(X_n) + g(\tau_{n+1}, X_{\tau_{n+1}}) (1- \tilde{f}(X_n))}\\
& \ge \sup_{f \in {\cal D}} \mathbb{E} \edg{g(n,X_n) f(X_n) + g(\tau_{n+1}, X_{\tau_{n+1}}) (1- f(X_n))} - \varepsilon/4.
\end{aligned}
\ee
$\tilde{f}$ can be written as $\tilde{f} = 1_A$ for the Borel set $A = \{x \in \mathbb{R}^d : \tilde{f}(x) =1\}$. 
Moreover, by \eqref{ic},
\[
B \mapsto \mathbb{E} \edg{|g(n,X_n)| 1_B(X_n)} \quad \mbox{and} \quad  
B \mapsto \mathbb{E} \edg{|g(\tau_{n+1}, X_{\tau_{n+1}})| 1_B(X_n)}
\]
define finite Borel measures on $\mathbb{R}^d$. Since every finite Borel measure on $\mathbb{R}^d$ is tight 
(see e.g., \cite{AB}), there exists a compact (possibly empty) subset $K \subseteq A$ such that
\be \label{K}
\begin{aligned}
& \mathbb{E} \edg{g(n,X_n) 1_K(X_n) + g(\tau_{n+1}, X_{\tau_{n+1}}) (1- 1_K(X_n))}\\
&\ge \mathbb{E} \edg{g(n,X_n) \tilde{f}(X_n) + g(\tau_{n+1}, X_{\tau_{n+1}}) (1- \tilde{f}(X_n))} - \varepsilon/4.
\end{aligned}
\ee
Let $\rho_K \colon \mathbb{R}^d \to [0,\infty]$ be the distance function given by
$\rho_K(x) = \inf_{y \in K} \|x-y\|_2$. Then 
\[
k_j(x) = \max \crl{1 - j\rho_K(x), -1}, \quad j \in \mathbb{N},
\]
defines a sequence of continuous functions $k_j \colon \mathbb{R}^d \to [-1,1]$ that converge pointwise to 
$1_K - 1_{K^c}$. So it follows from Lebesgue's dominated convergence theorem that there exists a $j \in \mathbb{N}$
such that
\be \label{kj}
\begin{aligned}
& \mathbb{E} \edg{g(n,X_n) \, 1_{\crl{k_j(X_n) \ge 0}}
+ g(\tau_{n+1}, X_{\tau_{n+1}}) (1- 1_{\crl{k_j (X_n) \ge 0}})}\\
&\ge \mathbb{E} \edg{g(n,X_n) 1_K(X_n) + g(\tau_{n+1}, X_{\tau_{n+1}}) (1- 1_K(X_n))} - \varepsilon/4.
\end{aligned}
\ee
By Theorem 1 of \cite{LLPS}, $k_j$ can be approximated uniformly on compacts by functions of the form
\be \label{h}
\sum_{i=1}^r (v_i^T x + c_i)^+ - \sum_{i=1}^s (w_i^T x + d_i)^+
\ee
for $r,s \in \mathbb{N}$, $v_1, \dots, v_r, w_1, \dots, w_s \in \mathbb{R}^d$ and $c_1, \dots, c_r,
d_1, \dots, d_s \in \mathbb{R}$. So there exists a function $h \colon \mathbb{R}^d \to \mathbb{R}$ 
expressible as in \eqref{h} such that
\be 
\begin{aligned} \label{hk}
&\mathbb{E} \edg{g(n,X_n) \, 1_{\crl{h(X_n) \ge 0}}
+ g(\tau_{n+1}, X_{\tau_{n+1}}) (1- 1_{\crl{h(X_n)\ge 0}})}\\
&\ge  \mathbb{E} \edg{g(n,X_n) \, 1_{\crl{k_j(X_n) \ge 0}}
+ g(\tau_{n+1}, X_{\tau_{n+1}}) (1- 1_{\crl{k_j (X_n) \ge 0}})}- \varepsilon/4.
\end{aligned}
\ee
Now note that for any integer $I \ge 2$, the composite mapping $1_{[0,\infty)} \circ h$ can 
be written as a neural net $f^{\theta}$ of the form \eqref{ftheta} with depth $I$ for suitable integers 
$q_1, \dots, q_{I-1}$ and parameter value $\theta \in \mathbb{R}^q$. Hence, one obtains 
from \eqref{fhat}, \eqref{K}, \eqref{kj} and \eqref{hk} that
\[
\begin{aligned}
&\mathbb{E} \edg{g(n,X_n) \, f^{\theta}(X_n) + g(\tau_{n+1}, X_{\tau_{n+1}}) (1- f^{\theta}(X_n))}\\
&\ge \sup_{f \in {\cal D}} \mathbb{E} \edg{g(n,X_n) f(X_n) 
+ g(\tau_{n+1}, X_{\tau_{n+1}}) (1- f (X_n))} - \varepsilon,
\end{aligned}
\]
and the proof is complete.
\end{proof}

We always choose $\theta_N \in \mathbb{R}^q$ such that\footnote{It is easy to see 
that this is possible.} $f^{\theta_N} \equiv 1$. Then our candidate optimal stopping time 
\be \label{tautheta}
\tau^{\Theta} = \sum_{n=1}^{N} n f^{\theta_n}(X_n) \prod_{j=0}^{n-1} (1-f^{\theta_j}(X_j))
\ee
is specified by the vector $\Theta = (\theta_0, \theta_1, \dots, \theta_{N-1}) \in \mathbb{R}^{Nq}$.
The following is an immediate consequence of Theorem \ref{thm:rep} and Proposition \ref{prop:appr}:

\begin{corollary}
For a given optimal stopping problem of the form \eqref{os}, a depth $I \ge 2$ and a constant $\varepsilon > 0$, there 
exist positive integers $q_1, \dots,  q_{I-1}$ and a vector $\Theta \in \mathbb{R}^{Nq}$ such that the 
corresponding stopping time \eqref{tautheta} satisfies
$
\mathbb{E} \, g(\tau^{\Theta}, X_{\tau^{\Theta}}) \ge \sup_{\tau \in {\cal T}} \mathbb{E} \, g(\tau, X_{\tau}) - \varepsilon.
$
\end{corollary}

\subsection{Parameter optimization}
\label{ss:parameter}

We train neural networks of the form \eqref{ftheta} with fixed depth $I \ge 2$ and 
given numbers $q_1, \dots, q_{I-1}$ of nodes in the hidden 
layers\footnote{For a given application, one can try out different choices of $I$ and 
$q_1, \dots, q_{I-1}$ to find a suitable trade-off between accuracy and efficiency. Alternatively, 
the determination of $I$ and $q_1, \dots, q_{I-1}$ could be built into the training algorithm.}.
To numerically find parameters $\theta_n \in \mathbb{R}^q$ yielding good stopping decisions $f^{\theta_n}$
for all times $n \in \crl{0,1, \dots, N-1}$, we approximate expected values with averages of 
Monte Carlo samples calculated from simulated paths of the process $(X_n)_{n=0}^N$. 

Let $(x^k_n)_{n=0}^N$, $k =1,2,\dots$ be independent realizations of such paths.
We choose $\theta_N \in \mathbb{R}^q$ such that $f^{\theta_N} \equiv 1$ and determine
determine $\theta_n \in \mathbb{R}^q$ for $n \le N-1$ recursively. So, suppose that for a given 
$n \in \crl{0,1, \dots, N-1}$, parameters $\theta_{n+1}, \dots, \theta_{N} \in \mathbb{R}^q$, have been found so that 
the stopping decisions $f^{\theta_{n+1}}, \dots, f^{\theta_N}$ generate a stopping time
\[
\tau_{n+1} = \sum_{m =n+1}^{N} mf^{\theta_m}(X_m) \prod_{j=n+1}^{m-1} (1-f^{\theta_j}(X_j))
\]
with corresponding expectation 
$
\mathbb{E}\, g( \tau_{ n + 1 } , X_{ \tau_{ n + 1 } } )
$ 
close to the optimal value $V_{n+1}$. If $n = N-1$, one has
$\tau_{n+1} \equiv N$, and if $n \le N-2$, $\tau_{n+1}$ can be written as
\[
\tau_{n+1} = l_{n+1}(X_{n+1}, \dots, X_{N-1}) 
\]
for a measurable function $l_{n+1} \colon \mathbb{R}^{d(N-n-1)} \to \crl{n+1, n+2, \dots, N}$. Accordingly, denote
\[
l^k_{n+1} = 
\begin{cases}
N & \mbox{if } n = N - 1 
\\
l_{n+1}(x^k_{n+1}, \dots, x^k_{N-1}) 
& \mbox{if } n \le N-2
\end{cases}
.
\]
If at time $n$, one applies the soft stopping decision $F^{\theta}$ and afterward behaves according to 
$f^{\theta_{n+1}}, \dots, f^{\theta_N}$, the realized reward along the $k$-th simulated path of $X$ is
\[
r^k_n(\theta) = g(n,x^k_n)F^{\theta}(x^k_n) + g(l^k_{n+1}, x^k_{l^k_{n+1}}) (1- F^{\theta}(x^k_n)).
\]
For large $K \in \mathbb{N}$, 
\be \label{Msum}
\frac{1}{K} \sum_{k =1}^{K} r^k_n(\theta)
\ee approximates the expected value
\[
\mathbb{E} \edg{g(n,X_n) F^{\theta}(X_n) + g(\tau_{n+1}, X_{\tau_{n+1}}) (1- F^{\theta}(X_n))}.
\]
Since $r^k_n(\theta)$ is almost everywhere differentiable in $\theta$, a stochastic gradient ascent method
can be applied to find an approximate optimizer $\theta_n \in \mathbb{R}^q$ of
\eqref{Msum}. The same simulations $(x^k_n)_{n=0}^N$, $k = 1,2,\dots$ can be used to train the stopping decisions 
$f^{\theta_n}$ at all times $n \in \crl{ 0, 1, \dots, N-1}$. In the numerical examples in Section \ref{sec:ex} below, we 
employed mini-batch gradient ascent with Xavier initialization \cite{GB10}, batch normalization 
\cite{IS15} and Adam updating \cite{KiB15}.

\begin{Remark} \label{Rem:I0}
If the Markov process $X$ starts from a deterministic initial value $x_0 \in \mathbb{R}^d$, the initial 
stopping decision is given by a constant $f_0 \in \crl{0,1}$. To learn $f_0$ from simulated paths 
of $X$, it is enough to compare the initial reward $g(0,x_0)$ to a Monte Carlo estimate 
$\hat{C}$ of $\mathbb{E} \, g(\tau_1,X_{\tau_1})$, where $\tau_1 \in {\cal T}_1$ is of the form 
\[
\tau_1 = \sum_{n=1}^{N} n f^{\theta_n}(X_n) \prod_{j=1}^{n-1} (1-f^{\theta_j}(X_j))
\]
for $f^{\theta_N} \equiv 1$ and trained parameters $\theta_1, \dots, \theta_{N-1} \in \mathbb{R}^q$.
Then one sets $f_0 = 1$ (that is, stop immediately) if $g(0,x_0) \ge \hat{C}$ and 
$f_0 = 0$ (continue) otherwise. The resulting stopping time is of the form 
\[
\tau^{\Theta} = 
\begin{cases}
0 & \mbox{if } f_0 = 1
\\
\tau_1 & \mbox{if } f_0 = 0.
\end{cases}
\]
\end{Remark}

\section{Bounds, point estimates and confidence intervals}
\label{sec:bci}

In this section we derive lower and upper bounds as well as point estimates and confidence intervals
for the optimal value $V_0 = \sup_{\tau \in {\cal T}} \, \mathbb{E} \, g(\tau, X_{\tau})$.

\subsection{Lower bound}
\label{subsec:lb}

Once the stopping decisions $f^{\theta_n}$ have been trained, the stopping time 
$\tau^{\Theta}$ given by \eqref{tautheta} yields a lower bound $L = \mathbb{E} \, g(\tau^{\Theta}, X_{\tau^{\Theta}})$ 
for the optimal value $V_0 = \sup_{\tau \in {\cal T}} \, \mathbb{E} \, g(\tau, X_{\tau})$. 
To estimate it, we simulate a new set\footnote{In particular, we assume that the samples
$(y^k_n)_{n=0}^N$, $k = 1, \dots, K_L$, are drawn independently from the 
realizations $(x^k_n)_{n=0}^N$, $k = 1, \dots, K$, used in the training of the stopping decisions.}
of independent realizations $(y^k_n)_{n=0}^N$, $k = 1,2,\dots, K_L,$ 
of $(X_n)_{n=0}^N$. $\tau^{\Theta}$ is of the form $\tau^{\Theta} = l(X_0,\dots, X_{N-1})$ for 
a measurable function $l \colon \mathbb{R}^{d N} \to \crl{0, 1, \dots, N}$. Denote $l^k = l(y^k_0, \dots, y^k_{N-1})$.
The Monte Carlo approximation 
\[
\hat{L} = \frac{1}{K_L} \sum_{k=1}^{K_L} g(l^k, y^k_{l^k})
\]
gives an unbiased estimate of the lower bound $L$, and by the law of large numbers, $\hat{L}$ converges to 
$L$ for $K_L \to \infty$. 

\subsection{Upper bound}

The Snell envelope of the reward process $(g(n,X_n))_{n=0}^N$ is the smallest\footnote{in the $\mathbb{P}$-almost sure order} 
supermartingale with respect to $({\cal F}_n)_{n=0}^N$ that dominates $(g(n,X_n))_{n=0}^N$. 
It is given\footnote{\label{fn}up to $\mathbb{P}$-almost sure equality} by 
\[
H_n = {\rm ess\,sup}_{\tau \in {\cal T}_n} \mathbb{E}[g(\tau) \mid {\cal F}_n], \quad n = 0, 1, \dots, N;
\] 
see, e.g., \cite{PS06, LL08}. Its Doob--Meyer decomposition is 
\[
H_n = H_0 + M^H_n - A^H_n,
\]
where $M^H$ is the $({\cal F}_n)$-martingale given\footref{fn} by
\[
M^H_0 = 0 \quad \mbox{and} \quad M^H_n - M^H_{n-1} = H_n - \mathbb{E}[H_n \mid {\cal F}_{n-1}], 
\quad n = 1, \dots, N,
\]
and $A^H$ is the nondecreasing $({\cal F}_n)$-predictable process given\footref{fn} by 
\[ A^H_0 = 0 \quad \mbox{and} \quad 
A^H_n - A^H_{n-1} = H_{n-1} - \mathbb{E}[H_n \mid {\cal F}_{n-1}], \quad n =1, \dots, N.
\]

Our estimate of an upper bound for the optimal value $V_0$ is based on the following 
variant\footnote{See also the discussion on noisy estimates in \cite{AB04}.
}
of the dual formulation of optimal stopping problems introduced by \cite{R02} and \cite{HK04}.

\begin{proposition}
Let $(\varepsilon_n)_{n=0}^N$ be a sequence of integrable random variables on 
$(\Omega, {\cal F}, \mathbb{P})$. Then 
\be \label{V0est1}
V_0 \ge \mathbb{E} \edg{ \max_{0 \le n \le N} \brak{g(n,X_n) - M^H_n - \varepsilon_n}}
+ \mathbb{E} \edg{ \min_{0 \le n \le N} \brak{A^H_n + \varepsilon_n}}.
\ee
Moreover, if $\mathbb{E} \edg{\varepsilon_n \mid {\cal F}_n} = 0$ for all $n \in \crl{0,1, \dots, N}$, one has
\be \label{V0est2}
V_0 \le \mathbb{E} \edg{ \max_{0 \le n \le N} \brak{g(n,X_n) - M_n - \varepsilon_n}}
\ee
for every $({\cal F}_n)$-martingale $(M_n)_{n=0}^N$ starting from $0$.
\end{proposition} 

\begin{proof}
First, note that 
\beas
&& \mathbb{E} \edg{ \max_{0 \le n \le N} \brak{g(n,X_n) - M^H_n - \varepsilon_n}}
\le \mathbb{E} \edg{ \max_{0 \le n \le N} \brak{H_n- M^H_n - \varepsilon_n}}\\
&&= \mathbb{E} \edg{ \max_{0 \le n \le N} \brak{H_0 - A^H_n - \varepsilon_n}}
= V_0 - \mathbb{E} \edg{ \min_{0 \le n \le N} \brak{A^H_n + \varepsilon_n }},
\eeas
which shows \eqref{V0est1}.

Now, assume that $\mathbb{E} \edg{\varepsilon_n \mid {\cal F}_n} = 0$ for all $n \in \crl{0,1, \dots, N}$,
and let $\tau$ be an $X$-stopping time.
Then 
\[
\mathbb{E} \, \varepsilon_{\tau} = \mathbb{E} \edg{\sum_{n=0}^N 1_{\crl{\tau = n}} \varepsilon_n}
= \mathbb{E} \edg{\sum_{n=0}^N 1_{\crl{\tau = n}} \mathbb{E}[\varepsilon_n \mid {\cal F}_n]} = 0.
\]
So one obtains from the optional stopping theorem (see, e.g., \cite{GS01}) that
\[
\mathbb{E} \, g(\tau,X_{\tau})
= \mathbb{E} \edg{g(\tau,X_{\tau}) - M_{\tau} - \varepsilon_{\tau}}
\le \mathbb{E} \edg{\max_{0 \le n \le N} \brak{g(n,X_n) - M_n - \varepsilon_n}}
\]
for every $({\cal F}_n)$-martingale $(M_n)_{n=0}^N$ starting from $0$.
Since $V_0 = \sup_{\tau \in {\cal T}} \mathbb{E} \, g(\tau,X_{\tau})$, this implies \eqref{V0est2}.
\end{proof}

For every $({\cal F}_n)$-martingale $(M_n)_{n=0}^N$ starting from $0$ and each sequence of integrable 
error terms $(\varepsilon_n)_{n=0}^N$ satisfying $\mathbb{E}\edg{\varepsilon_n \mid {\cal F}_n} = 0$
for all $n$, the right side of \eqref{V0est2} provides an upper bound\footnote{Note that 
for the right side of \eqref{V0est2} to be a valid upper bound, it is sufficient that 
$\mathbb{E}\edg{\varepsilon_n \mid {\cal F}_n} = 0$ for all $n$. In particular, 
$\varepsilon_0, \varepsilon_1, \dots, \varepsilon_N$ can have any arbitrary dependence structure.} for $V_0$, and by 
\eqref{V0est1}, this upper bound is tight if $M = M^H$ and $\varepsilon \equiv 0$.
So we try to use our candidate optimal stopping time $\tau^{\Theta}$ 
to construct a martingale close to $M^H$. The closer $\tau^{\Theta}$ is to an optimal stopping time, 
the better the value process\footnote{Again, since $H^{\Theta}_n$, $M^{\Theta}_n$ and $C^{\Theta}_n$ 
are given by conditional expectations, they are only specified up to $\mathbb{P}$-almost sure equality.}
\[
H^{\Theta}_n = \mathbb{E} \edg{g(\tau^{\Theta}_n,X_{\tau^{\Theta}_n}) \mid {\cal F}_n},
\quad n = 0, 1, \dots, N,
\]
corresponding to 
\[
\tau^{\Theta}_{n} = \sum_{m =n}^{N} mf^{\theta_m}(X_m) \prod_{j=n}^{m-1} (1-f^{\theta_j}(X_j)),
\quad n = 0,1,\dots, N,
\]
approximates the Snell envelope $(H_n)_{n=0}^N$. The martingale part of $(H^{\Theta}_n)_{n=0}^N$
is given by $M^{\Theta}_0 = 0$ and
\be \label{MTheta}
M^{\Theta}_n - M^{\Theta}_{n-1} = H^{\Theta}_n - \mathbb{E} \edg{H^{\Theta}_n \mid {\cal F}_{n-1}}
= f^{\theta_n}(X_n) g(n,X_n) + (1-f^{\theta_n}(X_n)) C^{\Theta}_n - C^{\Theta}_{n-1}, \;
n\ge 1,
\ee
for the continuation values\footnote{The two conditional expectations are equal since $(X_n)_{n=0}^N$ 
is Markov and $\tau^{\Theta}_{n+1}$ only depends on $(X_{n+1}, \dots, X_{N-1})$.}
\[
C^{\Theta}_n = \mathbb{E}[g(\tau^{\Theta}_{n+1}, X_{\tau^{\Theta}_{n+1}}) \mid {\cal F}_n] =
\mathbb{E}[g(\tau^{\Theta}_{n+1}, X_{\tau^{\Theta}_{n+1}}) \mid X_n], \quad n = 0,1, \dots, N-1.
\]
Note that $C^{\Theta}_N$ does not have to be specified. It formally appears in \eqref{MTheta} for $n = N$. 
But $(1-f^{\theta_N}(X_N))$ is always $0$.
To estimate $M^{\Theta}$, we generate a third set\footnote{The realizations $(z^k_n)_{n=0}^N$, $k = 1, \dots, K_U$, 
must be drawn independently of $(x^k_n)_{n=0}^N$, $k = 1, \dots, K$, so that our estimate of the upper bound 
does not depend on the samples used to train the stopping decisions. But theoretically, they can depend on 
$(y^k_n)_{n=0}^N$, $k = 1, \dots, K_L$, without affecting the unbiasedness of the estimate $\hat{U}$ or
the validity of the confidence interval derived in Subsection \ref{subsec:ci} below.} of independent realizations 
$(z^k_n)_{n=0}^N$, $k = 1,2,\dots, K_U,$ of $(X_n)_{n=0}^N$. In addition, for every $z^k_n$, we simulate
$J$ continuation paths $\tilde{z}^{k,j}_{n+1}, \dots, \tilde{z}^{k,j}_N$, $j =1, \dots, J$, 
that are conditionally independent\footnote{More precisely, the tuples 
$(\tilde{z}^{k,j}_{n+1}, \dots, \tilde{z}^{k,j}_N)$, $j = 1, \dots, J$, are simulated according to $p_n(z_n^k, \cdot)$, where
$p_n$ is a transition kernel from $\mathbb{R}^d$ to $\mathbb{R}^{(N-n)d}$ 
such that $p_n(X_n, B) =  \mathbb{P}[(X_{n+1}, \dots, X_N) \in B \mid X_n]$
$\mathbb{P}$-almost surely for all Borel sets $B \subseteq \mathbb{R}^{(N-n)d}$. 
We generate them independently of each other across $j$ and $k$. 
On the other hand, the continuation paths starting from $z^k_n$ do not have to be drawn independently
of those starting from $z^k_{n'}$ for $n \neq n'$. 
} 
of each other and of $z^{k}_{n+1}, \dots, z^{k}_N$.
Let us denote by $\tau^{k,j}_{n+1}$ the value of $\tau^{\Theta}_{n+1}$ along 
$\tilde{z}^{k,j}_{n+1}, \dots, \tilde{z}^{k,j}_N$. Estimating the continuation values as
\[
C^k_n = 
\frac{1}{J} \sum_{j = 1}^{J} g \brak{\tau^{k,j}_{n+1}, \tilde{z}^{k,j}_{\tau^{k,j}_{n+1}}}, \quad 
n = 0, 1, \dots, N-1,\]
yields the noisy estimates 
\[
\Delta M^k_n = f^{\theta_n}(z^k_n) g(n,z^k_n) + (1- f^{\theta_n}(z^k_n)) C^k_n - C^k_{n-1}
\]
of the increments $M^{\Theta}_n - M^{\Theta}_{n-1}$ along the $k$-th simulated path $z^k_0, \dots, z^k_N$.
So
\[
M^k_n = \begin{cases}
0 & \mbox{ if } n = 0\\
\sum_{m=1}^n \Delta M^k_m & \mbox{ if } n \ge 1
\end{cases}
\]
can be viewed as realizations of $M^{\Theta}_n + \varepsilon_n$ 
for estimation errors $\varepsilon_n$ with standard deviations proportional to 
$1/\sqrt{J}$ such that $\mathbb{E}\edg{\varepsilon_n \mid {\cal F}_n} = 0$ for all $n$. Accordingly,
\[
\hat{U} = \frac{1}{K_U} \sum_{k =1}^{K_U} \max_{0 \le n \le N} \brak{g \brak{n,z^k_n} - M^k_n},
\]
is an unbiased estimate of the upper bound 
\[
U = \mathbb{E} \edg{\max_{0 \le n \le N} \brak{g(n,X_n) - M^{\Theta}_n - \varepsilon_n}},
\]
which, by the law of large numbers, converges to $U$ for $K_U \to \infty$.

\subsection{Point estimate and confidence intervals}
\label{subsec:ci}

Our point estimate of $V_0$ is the average
\[
\frac{\hat{L} + \hat{U}}{2}.
\]
To derive confidence intervals, we assume that $g(n,X_n)$ is square-integrable\footnote{See condition \eqref{ic2}.}
for all $n$. Then
\[
g(\tau^{\theta}, X_{\tau^{\Theta}}) \quad \mbox{and} \quad 
\max_{0 \le n \le N} \brak{g(n,X_n) - M^{\Theta}_n - \varepsilon_n}
\]
are square-integrable too. Hence, one obtains from the central limit theorem that
for large $K_L$, $\hat{L}$ is approximately normally 
distributed with mean $L$ and variance $\hat{\sigma}^2_L/K_L$ for
\[
\hat{\sigma}^2_L = \frac{1}{K_L-1} \sum_{k=1}^{K_L} \brak{g(l^k,y^k_{l^k})-\hat{L}}^2.
\]
So, for every $\alpha \in (0,1]$, 
\[
\left[ \hat{L} - z_{\alpha/2} \frac{\hat{\sigma}_L}{\sqrt{K}_L} \, , \, \infty \right)
\]
is an asymptotically valid $1-\alpha/2$ confidence interval for $L$, where 
$z_{\alpha/2}$ is the $1-\alpha/2$ quantile of the standard normal distribution.
Similarly, 
\[
\left(-\infty \, , \, \hat{U} +  z_{\alpha/2} \frac{ \hat{\sigma}_U}{\sqrt{K}_U} \right] \quad \mbox{with} \quad
\hat{\sigma}^2_U = \frac{1}{K_U -1} \sum_{k=1}^{K_U} \brak{\max_{0 \le n \le N} \brak{g \brak{n,z^k_n} - M^k_n} - \hat{U}}^2,
\]
is an asymptotically valid $1- \alpha/2$ confidence interval for $U$. It follows that for every constant $\varepsilon > 0$,
one has
\beas
&& \mathbb{P} \edg{V_0 < \hat{L} - z_{\alpha/2} \frac{\hat{\sigma}_L}{\sqrt{K}_L} \; \; \mbox{ or } \;\;
V_0 > \hat{U} + z_{\alpha/2} \frac{\hat{\sigma}_U}{\sqrt{K}_U}}\\
&& \le \mathbb{P} \edg{L < \hat{L} - z_{\alpha/2} \frac{\hat{\sigma}_L}{\sqrt{K}_L}}
+ \mathbb{P} \edg{
U > \hat{U} + z_{\alpha/2} \frac{\hat{\sigma}_U}{\sqrt{K}_U}}
\le \alpha + \varepsilon
\eeas
as soon as $K_L$ and $K_U$ are large enough. In particular,
\be \label{ci}
\left[\hat{L} - z_{\alpha/2} \frac{\hat{\sigma}_L}{\sqrt{K}_L} \, , \,
\hat{U} +  z_{\alpha/2} \frac{\hat{\sigma}_U}{\sqrt{K}_U} \right]
\ee
is an asymptotically valid $1- \alpha$ confidence interval for $V_0$.

\section{Examples}
\label{sec:ex}

In this section we test\footnote{All computations were performed in single precision (float32) 
on a NVIDIA GeForce GTX 1080 GPU with 1974 MHz core clock and 8 GB GDDR5X memory 
with 1809.5 MHz clock rate. The underlying system consisted of an Intel Core i7-6800K 3.4 
GHz CPU with 64 GB DDR4-2133 memory running 
Tensorflow 1.11 on Ubuntu 16.04.} 
our method on three examples: the pricing of a Bermudan max-call option, 
the pricing of a callable multi barrier reverse convertible and 
the problem of optimally stopping a fractional Brownian motion.

\subsection{Bermudan max-call options}
\label{subsec:maxcall}

Bermudan max-call options are one of the most studied examples in the numerics literature on
optimal stopping problems; see, e.g., \cite{LS01,R02, G03, BKT03, HK04, BG04, AB04, BC08, 
BS08, Be11AA, Be13, JO15, Le16}. Their payoff depends on the maximum of $d$ underlying assets. 

Assume the risk-neutral dynamics of the assets are given by a multi-dimensional Black--Scholes
model\footnote{We make this assumption so that we can compare our results to those obtained 
with different methods in the literature. But our approach works for any asset dynamics as long 
as it can efficiently be simulated.}
\be \label{BS}
S^i_t = s^i_0 \exp\!\brak{[r-\delta_i - \sigma^2_i/2] t + \sigma_i W^i_t}, \quad 
i = 1, 2, \dots, d,
\ee
for initial values $ s^i_0 \in (0,\infty)$, a risk-free interest rate $r \in \mathbb{R}$,
dividend yields $\delta_i \in [0,\infty)$, volatilities $\sigma_i \in (0,\infty)$ and 
a $d$-dimensional Brownian motion $W$ with constant instantaneous correlations\footnote{That is,
$\mathbb{E}[(W^i_t-W^i_s)( W^j_t- W^i_s)] = \rho_{ij}(t-s)$ for all $i \neq j$ and $s < t$.} $\rho_{ij} \in \mathbb{R}$ 
between different components $W^i$ and $W^j$.
A Bermudan max-call option on $ S^1, S^2, \dots, S^d$ has payoff 
$\brak{\max_{1 \le i \le d} S^i_t - K}^+
$
and can be exercised at any point of a time grid $0=t_0 <t_1< \dots < t_N$. Its price is given by 
\[
\sup_{\tau} \mathbb{E}\!\left[ e^{ - r \tau } \left( 
\max_{ 1 \le i \le d } S^i_{ \tau } - K \right)^{ \! + } \right] ,
\]
where the supremum is over all $S$-stopping times taking values in $\crl{t_0,t_1, \dots, t_N}$; see, e.g., \cite{S02}. 
Denote $X^i_n = S^i_{t_n}$, $n = 0,1,\dots, N$, and let ${\cal T}$ be the set of $X$-stopping times. Then 
the price can be written as $\sup_{\tau \in {\cal T}} \mathbb{E} \, g(\tau, X_{\tau} )$ for 
\[
g(n,x) =  e^{-r t_n} \brak{\max_{1 \le i \le d} x^i - K}^+,
\]
and it is straight-forward to simulate $(X_n)_{n=0}^N$.

In the following we assume the time grid to be of the form $t_n = nT/N$, $n =0,1,\dots,N$, for a maturity 
$T > 0$ and $N+1$ equidistant exercise dates. Even though $g(n,X_n)$ does not carry any information 
that is not already contained in $X_n$, our method worked more efficiently when we trained the optimal 
stopping decisions on Monte Carlo simulations of the $d+1$-dimensional Markov process 
$(Y_n)_{n=0}^N = (X_n,g(n,X_n))_{n=0}^N$ instead of $(X_n)_{n=0}^N$. 
Since $Y_0$ is deterministic, we first trained stopping times $\tau_1 \in {\cal T}_1$ of the form
\[
\tau_1 = \sum_{n=1}^{N} n f^{\theta_n}(Y_n) \prod_{j=1}^{n-1} (1-f^{\theta_j}(Y_k))
\]
for $f^{\theta_N} \equiv 1$ and $f^{\theta_1}, \dots, f^{\theta_{N-1}} \colon \mathbb{R}^{d+1} \to \crl{0,1}$
given by \eqref{ftheta} with $I=3$ and $q_1 = q_2 = d+40$. Then we determined our
candidate optimal stopping times as 
\[
\tau^{\Theta} = 
\begin{cases}
0 & \mbox{if } f_0 = 1\\
\tau_1 & \mbox{if } f_0 = 0
\end{cases}
\]
for a constant $f_0 \in \crl{0,1}$ depending\footnote{In fact, in none of the examples 
in this paper it is optimal to stop at time $0$. So $\tau^{\Theta} = \tau_1$ in all these cases.} 
on whether it was optimal to stop immediately at time $0$ or not (see Remark \ref{Rem:I0} above).

It is straight-forward to simulate from model \eqref{BS}.
We conducted 3,00$0+d$ training steps, in each of which we generated a batch of 
8,192 paths of $(X_n)_{n=0}^N$. To estimate the lower bound $L$ we simulated $K_L = 4,$096,000 trial paths.
For our estimate of the upper bound $U$, we produced $K_U = 1$,024 paths $(z^k_n)_{n = 0}^N$,
$k=1, \dots, K_U$, of $(X_n)_{n=0}^N$ and $K_U \times J$ realizations 
$(v^{k,j}_n)_{n=1}^N$, $k = 1, \dots, K_U$, $j =1, \dots, J$, of $(W_{t_n}- W_{t_{n-1}})_{n=1}^N$
with $J = 16,$384. Then for all $n$ and $k$, we generated the $i$-th component of the $j$-th continuation 
path departing from $z^k_n$ according to
\[
\tilde{z}^{i,k,j}_m = z^{i,k}_n
\exp \brak{[r- \delta_i - \sigma^2_i/2] (m-n) \Delta t + \sigma_i[v^{i,k,j}_{n+1} + \dots + v^{i,k,j}_m]}, \quad m = n+1, \dots, N.
\]

\medskip
{\bf Symmetric case}\\
We first considered the special case, where $s^i_0 = s_0$, $\delta_i = \delta$, $\sigma_i = \sigma$
for all $i =1, \dots, d,$ and $\rho_{ij} = \rho$ for all $i \neq j$. Our results are reported in 
Table~\ref{table:symm}.

\medskip
{\bf Asymmetric case}\\
As a second example, we studied model \eqref{BS} with
$s^i_0 = s_0$, $\delta_i = \delta$ for all $i = 1, 2, \dots, d,$ and $\rho_{ij} = \rho$ for
all $i \neq j$, but different volatilities $\sigma_1 < \sigma_2 < \dots < \sigma_d$. 
For $d \le 5$, we chose the specification $\sigma_i = 0.08 + 0.32 \times (i-1)/(d-1)$, $i=1,2, \dots, d$. 
For $d > 5$, we set $\sigma_i = 0.1 + i/(2d)$, $i = 1,2, \dots, d$. The results are given in 
Table~\ref{table:asymm}.

\begin{table}
\centering
\begin{small}
\begin{tabular}{c c c c c c c c c c} 
 \hline \\[-3mm]
  $d$ & $s_0$ & $\hat{L}$ & $t_L$ & $\hat{U}$ & $t_U$ & Point est.\ & $95\%$ CI & Binomial & BC $95\%$ CI\\[1mm]
 \hline \\[-3mm]
 $2$ & $90$ & $8.072$ & $28.7$ & $8.075$ & $25.4$ & $8.074$ & $[8.060, 8.081]$ & $8.075$ &\\
 $2$ & $100$ & $13.895$ & $28.7$ & $13.903$ & $25.3$ & $13.899$ & $[13.880, 13.910]$ & $13.902$ &\\
 $2$ & $110$ & $21.353$ & $28.4$ & $21.346$ & $25.3$ & $21.349$ & $[21.336, 21.354]$ & $21.345$ &\\[1mm]
 $3$ & $90$ & $11.290$ & $28.8$ & $11.283$ & $26.3$ & $11.287$ & $[11.276, 11.290]$ & $11.29$ &\\
 $3$ & $100$ & $18.690$ & $28.9$ & $18.691$ & $26.4$ & $18.690$ & $[18.673, 18.699]$ & $18.69$ &\\
 $3$ & $110$ & $27.564$ & $27.6$ & $27.581$ & $26.3$ & $27.573$ & $[27.545, 27.591]$ & $27.58$&\\[1mm]
 $5$ & $90$ & $16.648$ & $27.6$ & $16.640$ & $28.4$ & $16.644$ & $[16.633, 16.648]$ & & $[16.620, 16.653]$\\
 $5$ & $100$ & $26.156$ & $28.1$ & $26.162$ & $28.3$ & $26.159$ & $[26.138, 26.174]$ & & $[26.115, 26.164]$\\
 $5$ & $110$ & $36.766$ & $27.7$ & $36.777$ & $28.4$ & $36.772$ & $[36.745, 36.789]$ & & $[36.710, 36.798]$\\[1mm]
 $10$ & $90$ & $26.208$ & $30.4$ & $26.272$ & $33.9$ & $26.240$ & $[26.189, 26.289]$ & &\\
 $10$ & $100$ & $38.321$ & $30.5$ & $38.353$ & $34.0$ & $38.337$ & $[38.300, 38.367]$ & &\\
 $10$ & $110$ & $50.857$ & $30.8$ & $50.914$ & $34.0$ & $50.886$ & $[50.834, 50.937]$ & &\\[1mm]
 $20$ & $90$ & $37.701$ & $37.2$ & $37.903$ & $44.5$ & $37.802$ & $[37.681, 37.942]$ & &\\
 $20$ & $100$ & $51.571$ & $37.5$ & $51.765$ & $44.3$ & $51.668$ & $[51.549, 51.803]$ & &\\
 $20$ & $110$ & $65.494$ & $37.3$ & $65.762$ & $44.4$ & $65.628$ & $[65.470, 65.812]$ & &\\[1mm]
 $30$ & $90$ & $44.797$ & $45.1$ & $45.110$ & $56.2$ & $44.953$ & $[44.777, 45.161]$ & &\\
 $30$ & $100$ & $59.498$ & $45.5$ & $59.820$ & $56.3$ & $59.659$ & $[59.476, 59.872]$ & &\\
 $30$ & $110$ & $74.221$ & $45.3$ & $74.515$ & $56.2$ & $74.368$ & $[74.196, 74.566]$ & &\\[1mm]
 $50$ & $90$ & $53.903$ & $58.7$ & $54.211$ & $79.3$ & $54.057$ & $[53.883, 54.266]$ & &\\
 $50$ & $100$ & $69.582$ & $59.1$ & $69.889$ & $79.3$ & $69.736$ & $[69.560, 69.945]$ & &\\
 $50$ & $110$ & $85.229$ & $59.0$ & $85.697$ & $79.3$ & $85.463$ & $[85.204, 85.763]$ & &\\[1mm]
 $100$ & $90$ & $66.342$ & $95.5$ & $66.771$ & $147.7$ & $66.556$ & $[66.321, 66.842]$ & &\\
 $100$ & $100$ & $83.380$ & $95.9$ & $83.787$ & $147.7$ & $83.584$ & $[83.357, 83.862]$ & &\\
 $100$ & $110$ & $100.420$ & $95.4$ & $100.906$ & $147.7$ & $100.663$ & $[100.394, 100.989]$ & &\\[1mm]
 $200$ & $90$ & $78.993$ & $170.9$ & $79.355$ & $274.6$ & $79.174$ & $[78.971, 79.416]$ & &\\
 $200$ & $100$ & $97.405$ & $170.1$ & $97.819$ & $274.3$ & $97.612$ & $[97.381, 97.889]$ & &\\
 $200$ & $110$ & $115.800$ & $170.6$ & $116.377$ & $274.5$ & $116.088$ & $[115.774, 116.472]$ & &\\[1mm]
 $500$ & $90$ & $95.956$ & $493.4$ & $96.337$ & $761.2$ & $96.147$ & $[95.934, 96.407]$ & &\\
 $500$ & $100$ & $116.235$ & $493.5$ & $116.616$ & $761.7$ & $116.425$ & $[116.210, 116.685]$ & &\\
 $500$ & $110$ & $136.547$ & $493.7$ & $136.983$ & $761.4$ & $136.765$ & $[136.521, 137.064]$ & &\\[1mm]
 \hline
\end{tabular}
\caption{\label{table:symm}Summary results for max-call options on $d$ symmetric assets for parameter values of
$r = 5\%$, $\delta = 10\%$, $\sigma = 20\%$, $\rho = 0$, $K=100$, $T=3$, $N=9$.
$t_L$ is the number of seconds it took to train $\tau^{\Theta}$ and compute $\hat{L}$.
$t_U$ is the computation time for $\hat{U}$ in seconds.
95\% CI is the 95\% confidence interval \eqref{ci}. The binomial values were calculated 
with a binomial lattice method in \cite{AB04}. BC 95\% CI is the 95\% confidence interval computed in \cite{BC08}.}
\end{small}
\end{table}

\begin{table} 
\centering
\begin{small}
\begin{tabular}{c c c c c c c c c} 
 \hline \\[-3mm]
  $d$ & $s_0$ & $\hat{L}$ & $t_L$ & $\hat{U}$ & $t_U$ & Point est.\ & $95\%$ CI & BC $95\%$ CI\\[1mm]
 \hline \\[-3mm]
 $2$ & $90$ & $14.325$ & $26.8$ & $14.352$ & $25.4$ & $14.339$ & $[14.299, 14.367]$ &\\
 $2$ & $100$ & $19.802$ & $27.0$ & $19.813$ & $25.5$ & $19.808$ & $[19.772, 19.829]$ &\\
 $2$ & $110$ & $27.170$ & $26.5$ & $27.147$ & $25.4$ & $27.158$ & $[27.138, 27.163]$ &\\[1mm]
 $3$ & $90$ & $19.093$ & $26.8$ & $19.089$ & $26.5$ & $19.091$ & $[19.065, 19.104]$ &\\
 $3$ & $100$ & $26.680$ & $27.5$ & $26.684$ & $26.4$ & $26.682$ & $[26.648, 26.701]$ &\\
 $3$ & $110$ & $35.842$ & $26.5$ & $35.817$ & $26.5$ & $35.829$ & $[35.806, 35.835]$ &\\[1mm]
 $5$ & $90$ & $27.662$ & $28.0$ & $27.662$ & $28.6$ & $27.662$ & $[27.630, 27.680]$ & $[27.468, 27.686]$\\
 $5$ & $100$ & $37.976$ & $27.5$ & $37.995$ & $28.6$ & $37.985$ & $[37.940, 38.014]$ & $[37.730, 38.020]$\\
 $5$ & $110$ & $49.485$ & $28.2$ & $49.513$ & $28.5$ & $49.499$ & $[49.445, 49.533]$ & $[49.155, 49.531]$\\[1mm]
 $10$ & $90$ & $85.937$ & $31.8$ & $86.037$ & $34.4$ & $85.987$ & $[85.857, 86.087]$ &\\
 $10$ & $100$ & $104.692$ & $30.9$ & $104.791$ & $34.2$ & $104.741$ & $[104.603, 104.864]$ &\\
 $10$ & $110$ & $123.668$ & $31.0$ & $123.823$ & $34.4$ & $123.745$ & $[123.570, 123.904]$ &\\[1mm]
 $20$ & $90$ & $125.916$ & $38.4$ & $126.275$ & $45.6$ & $126.095$ & $[125.819, 126.383]$ &\\
 $20$ & $100$ & $149.587$ & $38.2$ & $149.970$ & $45.2$ & $149.779$ & $[149.480, 150.053]$ &\\
 $20$ & $110$ & $173.262$ & $38.4$ & $173.809$ & $45.3$ & $173.536$ & $[173.144, 173.937]$ &\\[1mm]
 $30$ & $90$ & $154.486$ & $46.5$ & $154.913$ & $57.5$ & $154.699$ & $[154.378, 155.039]$ &\\
 $30$ & $100$ & $181.275$ & $46.4$ & $181.898$ & $57.5$ & $181.586$ & $[181.155, 182.033]$ &\\
 $30$ & $110$ & $208.223$ & $46.4$ & $208.891$ & $57.4$ & $208.557$ & $[208.091, 209.086]$ &\\[1mm]
 $50$ & $90$ & $195.918$ & $60.7$ & $196.724$ & $81.1$ & $196.321$ & $[195.793, 196.963]$ &\\
 $50$ & $100$ & $227.386$ & $60.7$ & $228.386$ & $81.0$ & $227.886$ & $[227.247, 228.605]$ &\\
 $50$ & $110$ & $258.813$ & $60.7$ & $259.830$ & $81.1$ & $259.321$ & $[258.661, 260.092]$ &\\[1mm]
 $100$ & $90$ & $263.193$ & $98.5$ & $264.164$ & $151.2$ & $263.679$ & $[263.043, 264.425]$ &\\
 $100$ & $100$ & $302.090$ & $98.2$ & $303.441$ & $151.2$ & $302.765$ & $[301.924, 303.843]$ &\\
 $100$ & $110$ & $340.763$ & $97.8$ & $342.387$ & $151.1$ & $341.575$ & $[340.580, 342.781]$ &\\[1mm]
 $200$ & $90$ & $344.575$ & $175.4$ & $345.717$ & $281.0$ & $345.146$ & $[344.397, 346.134]$ &\\
 $200$ & $100$ & $392.193$ & $175.1$ & $393.723$ & $280.7$ & $392.958$ & $[391.996, 394.052]$ &\\
 $200$ & $110$ & $440.037$ & $175.1$ & $441.594$ & $280.8$ & $440.815$ & $[439.819, 441.990]$ &\\[1mm]
 $500$ & $90$ & $476.293$ & $504.5$ & $477.911$ & $760.7$ & $477.102$ & $[476.069, 478.481]$ &\\
 $500$ & $100$ & $538.748$ & $504.6$ & $540.407$ & $761.6$ & $539.577$ & $[538.499, 540.817]$ &\\
 $500$ & $110$ & $601.261$ & $504.9$ & $603.243$ & $760.8$ & $602.252$ & $[600.988, 603.707]$ &\\[1mm]
 \hline
\end{tabular}
\caption{\label{table:asymm}Summary results for max-call options on $d$ asymmetric assets for parameter values of
$r = 5\%$, $\delta = 10\%$, $\rho = 0$, $K=100$, $T=3$, $N=9$.
$t_L$ is the number of seconds it took to train $\tau^{\Theta}$ and compute $\hat{L}$.
$t_U$ is the computation time for $\hat{U}$ in seconds.
95\% CI is the 95\% confidence interval \eqref{ci}. BC 95\% CI is 
the 95\% confidence interval computed in \cite{BC08}.}
\end{small}
\end{table}

\subsection{Callable multi barrier reverse convertibles}
\label{subsec:callable}

A MBRC is a coupon paying security that converts into shares of the worst-performing
of $d$ underlying assets if a prespecified trigger event occurs.
Let us assume that the price of the $i$-th underlying asset in percent of its starting value
follows the risk-neutral dynamics 
\be \label{BS2}
S^i_t = 
\begin{cases}100 \exp\!\brak{[r - \sigma^2_i/2] t + \sigma_i W^i_t} & \mbox{for } t \in [0 , T_i)\\
100 (1-\delta_i) \exp\!\brak{[r - \sigma^2_i/2] t + \sigma_i W^i_t} & \mbox{for } t \in [T_i,T]
\end{cases}
\ee
for a risk-free interest rate $r \in \mathbb{R}$, volatility $\sigma_i \in (0,\infty)$, maturity $T \in (0,\infty)$,
dividend payment time $T_i \in (0,T)$, dividend rate $\delta_i \in [0,\infty)$
and a $d$-dimensional Brownian motion $W$ with constant instantaneous 
correlations $\rho_{ij} \in \mathbb{R}$ between different components $W^i$ and $W^j$.

Let us consider a MBRC that pays a coupon $c$ at each of $N$ time points $t_n = nT/N$, $n = 1, 2, \dots, N$, 
and makes a time-$T$ payment of
\[
G = \begin{cases}
F & \mbox{ if } \min_{1 \le i \le d} \min_{1 \le m \le M}  S^i_{u_m} > B \mbox{ or } \min_{1 \le i \le d} S^i_T > K\\
\min_{1 \le i \le d} S^i_T & \mbox{ if } \min_{1 \le i \le d} \min_{1 \le m \le M} S^i_{u_m} \le B \mbox{ and } \min_{1 \le i \le d} S^i_T \le K,
\end{cases}
\]
where $F \in [0,\infty)$ is the nominal amount, $B \in [0,\infty)$ a barrier,
$K \in [0,\infty)$ a strike price and $u_m$ the end of the $m$-th trading day.
Its value is 
\be \label{payoff}
\sum_{n=1}^N e^{-rt_n} c + e^{-rT} \mathbb{E} \, G
\ee
and can easily be estimated with a standard Monte Carlo approximation.

A callable MBRC can be redeemed by the issuer at any of the times 
$t_1, t_2, \dots, t_{N-1}$ by paying back the notional. To minimize costs, the issuer will 
try to find a $\crl{t_1, t_2, \dots, T}$-valued stopping time such that
\[
\mathbb{E} \edg{\sum_{n=1}^{\tau} e^{-r t_n} c + 1_{\crl{\tau < T}} e^{-r \tau} F + 1_{\crl{\tau = T}} e^{-r T} G}
\]
is minimal. 

Let $(X_n)_{n=1}^N$ be the $d+1$-dimensional Markov process given by 
$X^i_n = S^i_{t_n}$ for $i = 1, \dots, d$, and 
\[
X^{d+1}_n := \begin{cases} 
1 & \mbox{ if the barrier has been breached before or at time $t_n$}\\
0 & \mbox{ else}.
\end{cases}
\]
Then the issuer's minimization problem can be written as 
\be \label{min}
\inf_{\tau \in {\cal T}} \mathbb{E} \, g(\tau, X_{\tau}),
\ee
where ${\cal T}$ is the set of all $X$-stopping times and 
\[
g(n,x) = \begin{cases}
\sum_{m=1}^n e^{-r t_m} c + e^{-rt_n} F & \mbox{ if } 1 \le n \le N-1 \mbox{ or } x^{d+1}=0\\
\sum_{m=1}^N e^{-r t_m} c + e^{-r t_N} h(x) & \mbox{ if } n = N \mbox{ and } x^{d+1} = 1,
\end{cases}
\]
where
\[
h(x) = \begin{cases} 
F & \mbox{ if } \min_{1 \le i \le d} x^i > K\\
\min_{1 \le i \le d} x^i & \mbox{ if } \min_{1 \le i \le d} x^i \le K.
\end{cases}
\]
Since the issuer cannot redeem at time $0$, we trained stopping times of the form
\[
\tau^{\Theta} = \sum_{n=1}^{N} n f^{\theta_n}(Y_n) \prod_{j=1}^{n-1} (1-f^{\theta_j}(Y_k)) \in {\cal T}_1
\]
for $f^{\theta_N} \equiv 1$ and $f^{\theta_1}, \dots, f^{\theta_{N-1}} \colon \mathbb{R}^{d+1} \to \crl{0,1}$
given by \eqref{ftheta} with $I=3$ and $q_1 = q_2 = d+ 40$. Since \eqref{min} is a minimization problem, 
$\tau^{\Theta}$ yields an upper bound and the dual method a lower bound.

We simulated the model \eqref{BS2} like \eqref{BS} in Subsection \ref{subsec:maxcall} with the same 
number of trials except that here we used the lower number $J = 1,$024 to estimate the dual bound.
Numerical results are reported in Table~\ref{table:rbc}.

\begin{table} 
\centering
\begin{small}
\begin{tabular}{c c c c c c c c c} 
 \hline \\[-3mm]
 $d$ & $\rho$ & $\hat{L}$ & $t_L$ & $\hat{U}$ & $t_U$ & Point est.\ & $95\%$ CI & Non-callable\ \\[1mm]
 \hline \\[-3mm]
 $2$ & $0.6$ & $98.235$ & $24.9$ & $98.252$ & $204.1$ & $98.243$ & $[98.213,98.263]$ & $106.285$ \\
 $2$ & $0.1$ & $97.634$ & $24.9$ & $97.634$ & $198.8$ & $97.634$ & $[97.609,97.646]$ & $106.112$ \\
 $3$ & $0.6$ & $96.930$ & $26.0$ & $96.936$ & $212.9$ & $96.933$ & $[96.906,96.948]$ & $105.994$ \\
 $3$ & $0.1$ & $95.244$ & $26.2$ & $95.244$ & $211.4$ & $95.244$ & $[95.216,95.258]$ & $105.553$ \\
 $5$ & $0.6$ & $94.865$ & $41.0$ & $94.880$ & $239.2$ & $94.872$ & $[94.837,94.894]$ & $105.530$ \\
 $5$ & $0.1$ & $90.807$ & $41.1$ & $90.812$ & $238.4$ & $90.810$ & $[90.775,90.828]$ & $104.496$ \\
 $10$ & $0.6$ & $91.568$ & $71.3$ & $91.629$ & $300.9$ & $91.599$ & $[91.536,91.645]$ & $104.772$ \\
 $10$ & $0.1$ & $83.110$ & $71.7$ & $83.137$ & $301.8$ & $83.123$ & $[83.078,83.153]$ & $102.495$ \\
 $15$ & $0.6$ & $89.558$ & $94.9$ & $89.653$ & $359.8$ & $89.606$ & $[89.521,89.670]$ & $104.279$ \\
 $15$ & $0.1$ & $78.495$ & $94.7$ & $78.557$ & $360.5$ & $78.526$ & $[78.459,78.571]$ & $101.209$ \\
 $30$ & $0.6$ & $86.089$ & $158.5$ & $86.163$ & $534.1$ & $86.126$ & $[86.041,86.180]$ & $103.385$\\
 $30$ & $0.1$ & $72.037$ & $159.3$ & $72.749$ & $535.6$ & $72.393$ & $[71.830,72.760]$ & $99.279$\\[1mm]
 \hline
\end{tabular}
\caption{\label{table:rbc}Summary results for callable MBRCs with $d$ underlying assets for
$F = K = 100$, $B = 70$, $T = 1$ year ($= 252$ trading days), $N = 12$, $c = 7/12$, $\delta_i = 5\%$, 
$T_i = 1/2$, $r = 0$, $\sigma_i = 0.2$ and $\rho_{ij} = \rho$ for $i \neq j$. 
$t_U$ is the number of seconds it took to train $\tau^{\Theta}$ and compute $\hat{U}$.
$t_L$ is the number of seconds it took to compute $\hat{L}$. The last column lists 
fair values of the same MBRCs without the callable feature. We estimated them 
by averaging 4,096,000 Monte Carlo samples of the payoff. This took between 5 (for $d=2$) and 
44 (for $d$ = 30) seconds.}
\end{small}
\end{table}

\subsection{Optimally stopping a fractional Brownian motion}
\label{subsec:fBm}

A fractional Brownian motion with Hurst parameter $ H \in (0,1] $ 
is a continuous centered Gaussian process $ ( W^H_t )_{ t \ge 0 } $ with covariance structure 
\[
\mathbb{E}[W^H_t W^H_s] = \frac{1}{2} \brak{t^{2H} + s^{2H} - |t-s|^{2H}};
\]
see, e.g., \cite{MVN,ST}. For $ H = 1/2$, $W^H$ is a standard Brownian motion. So, by the 
optional stopping theorem, one has $\mathbb{E}\, W^{1/2}_{ \tau }= 0$ for every $W^{1/2}$-stopping time $\tau$
bounded above by a constant; see, e.g., \cite{GS01}. However, 
for $ H \neq 1/2$, the increments of $W^H$ are correlated --
positively for $ H \in (1/2, 1 ]$ and negatively for $ H \in ( 0, 1/2) $. 
In both cases, $W^H$ is neither a martingale nor a Markov process,
and there exist bounded $W^H $-stopping times $\tau$ such that $ \mathbb{E}\, W^H_{ \tau }> 0$;
see, e.g., \cite{KG16} for two classes of simple stopping rules $ 0 \leq \tau \le 1 $ 
and estimates of the corresponding expected values $\mathbb{E} \, W^H_{ \tau }$.

To approximate the supremum 
\be \label{stfBm}
\sup_{ 0 \leq \tau \le 1 } 
\mathbb{E}\, W^H_{\tau} 
\ee over all $W^H$-stopping times $ 0 \leq \tau \le 1 $,
we denote $ t_n = n/100$, $n = 0, 1, 2, \dots, 100$, and introduce the 
$100$-dimensional Markov process $(X_n)_{n=0}^{100}$ given by 
\[
\begin{split}
X_0 &= (0,0,\dots, 0)\\
X_1 &= (W^H_{t_1},0, \dots, 0)\\
X_2 &= (W^H_{t_2}, W^H_{t_1}, 0, \dots, 0)\\
\vdots & \\
X_{100} &= (W^H_{t_{100}}, W^H_{t_{99}}, \dots, W^H_{t_1}).
\end{split}
\]
The discretized stopping problem 
\be \label{disfBm}
\sup_{ \tau \in {\cal T} } \mathbb{E} \, g(X_{\tau}),
\ee
where ${\cal T}$ is the set of all $X$-stopping times and $g \colon \mathbb{R}^{100} \to \mathbb{R}$
the projection $(x^1, \dots, x^{100}) \mapsto x^1$, approximates \eqref{stfBm} from below. 

We computed estimates of \eqref{disfBm} for $H \in \{ 0.01, 0.05, 0.1, 0.15, \dots, 1 \}$ by 
training networks of the form \eqref{ftheta} with depth $I = 3$, $d = 100 $ and $q_1 = q_2 = 140$. 
To simulate the vector $Y = (W^H_{t_n})_{n=0}^{100}$, we used the representation $Y = B Z$, 
where $BB^T$ is the Cholesky decomposition of the covariance matrix 
of $Y$ and $Z$ a $100$-dimensional random vector with independent standard normal components.
We carried out 6,000 training steps with a batch size of 2,048. To estimate the lower bound $L$ we 
generated $K_L = 4,$096,000 simulations of $Z$. For our estimate of the upper bound $U$, we first 
simulated $K_U = 1,$024 realizations $v^k$, $k = 1, \dots, K_U$ of $Z$ and set $w^k = B v^k$.
Then we produced another $K_U \times J$ simulations $\tilde{v}^{k,j}$, 
$k=1, \dots, K_U$, $j=1, \dots, J$, of $Z$, and generated for all $n$ and $k$, continuation paths starting from 
\[
z^k_n = (w^k_n, \dots, w^k_1, 0 , \dots, 0)
\]
according to 
\[
\tilde{z}^{k,j}_m
= (\tilde{w}^{k,j}_m, \dots, \tilde{w}^{k,j}_{n+1}, w^k_n, \dots, w^k_1, 0 \dots,0), \quad
m = n+1, \dots, 100,
\]
with
\[
\tilde{w}^{k,j}_l = \sum_{i=1}^n B_{li} v^k_i + \sum_{i=n+1}^l B_{li} \tilde{v}^{k,j}_i, \quad 
l = n+1, \dots, m.
\]
For $H \in \crl{0.01,...,0.4} \cup \crl{0.6,...,1.0}$, we chose $J = 16,$384, and for 
$H \in \crl{0.45,0.5,0.55}$, $J = 32,$768. The results are listed in Table~\ref{table:fbm}
and depicted in graphical form in Figure~\ref{fig:fbm}. 
Note that for $H = 1/2$ and $H=1$, our 95\% confidence intervals contain the true values, which 
in these two cases, can be calculated exactly. As mentioned above, $W^{1/2}$ is a Brownian motion, and therefore,
$\mathbb{E} \, W^{1/2}_{\tau} = 0$ for every $(W^{1/2}_{t_n})_{n = 0}^{100}$-stopping time $\tau$.
On the other hand, one has\footnote{\label{fn2}up to $\mathbb{P}$-almost sure equality} $W^1_t = t W^1_1$, $t \ge 0$. 
So, in this case, the optimal stopping time is 
given\footref{fn2} by 
\[
\tau = \begin{cases}
1 & \mbox{ if } W^1_{t_1} > 0\\
t_1 & \mbox{ if } W^1_{t_1} \le 0,
\end{cases}
\]
and the corresponding expectation by 
\[
\mathbb{E} \, W^1_{\tau} = \mathbb{E} \edg{W^1_1 1_{\crl{W^1_{t_1} > 0}} - W^1_{t_1} 1_{\crl{W^1_{t_1} \le 0}}}
= 0.99 \, \mathbb{E} \edg{W^1_1 1_{\crl{W^1_1 > 0}}} = 0.99/\sqrt{2 \pi} = 0.39495...
\]
Moreover, it can be seen that for $H \in (1/2,1)$, our estimates are up to
three times higher than the expected payoffs generated by the heuristic stopping rules of \cite{KG16}.
For $H \in(0, 1/2)$, they are up to five times higher. 

\begin{table}
\centering
\begin{small}
\begin{tabular}{c c c c c c c} 
 \hline \\[-3mm]
 $H$ & $\hat{L}$ & $\hat{U}$ & Point est.\ & $95\%$ CI\\[1mm]
 \hline \\[-3mm]
 $0.01$ & $1.518$ & $1.519$ & $1.519$ & $[1.517,1.520]$\\
 $0.05$ & $1.293$ & $1.293$ & $1.293$ & $[1.292,1.294]$\\
 $0.10$ & $1.048$ & $1.049$ & $1.049$ & $[1.048,1.050]$\\
 $0.15$ & $0.838$ & $0.839$ & $0.839$ & $[0.838,0.840]$\\
 $0.20$ & $0.658$ & $0.659$ & $0.658$ & $[0.657,0.659]$\\
 $0.25$ & $0.501$ & $0.504$ & $0.503$ & $[0.501,0.505]$\\
 $0.30$ & $0.369$ & $0.370$ & $0.370$ & $[0.368,0.371]$\\
 $0.35$ & $0.255$ & $0.256$ & $0.255$ & $[0.254,0.257]$\\
 $0.40$ & $0.155$ & $0.158$ & $0.156$ & $[0.154,0.158]$\\
 $0.45$ & $0.067$ & $0.075$ & $0.071$ & $[0.066,0.075]$\\
 $0.50$ & $0.000$ & $0.005$ & $0.002$ & $[0.000,0.005]$\\
 $0.55$ & $0.057$ & $0.065$ & $0.061$ & $[0.057,0.065]$\\
 $0.60$ & $0.115$ & $0.118$ & $0.117$ & $[0.115,0.119]$\\
 $0.65$ & $0.163$ & $0.165$ & $0.164$ & $[0.163,0.166]$\\
 $0.70$ & $0.206$ & $0.207$ & $0.207$ & $[0.205,0.208]$\\
 $0.75$ & $0.242$ & $0.245$ & $0.244$ & $[0.242,0.245]$\\
 $0.80$ & $0.276$ & $0.278$ & $0.277$ & $[0.276,0.279]$\\
 $0.85$ & $0.308$ & $0.309$ & $0.308$ & $[0.307,0.310]$\\
 $0.90$ & $0.336$ & $0.339$ & $0.337$ & $[0.335,0.339]$\\
 $0.95$ & $0.365$ & $0.367$ & $0.366$ & $[0.365,0.367]$\\
 $1.00$ & $0.395$ & $0.395$ & $0.395$ & $[0.394,0.395]$\\[1mm]
 \hline
\end{tabular}
\caption{ \label{table:fbm}Estimates of $\sup_{\tau \in \crl{0,t_1, \dots, 1}} \mathbb{E}\, W^H_{ \tau }$. For all 
$H \in \crl{0.01, 0.05, \dots, 1}$, it took about 430 seconds to train $\tau^{\Theta}$ and compute $\hat{L}$.
The computation of $\hat{U}$ took about 17,000 seconds for $H \in \crl{0.01, \dots, 0.4} \cup \crl{0.6, \dots, 1}$ and
about 34,000 seconds for $H \in \crl{0.45,0.5,0.55}$.}
\end{small}
\end{table}

\begin{figure} 
\centering
\includegraphics{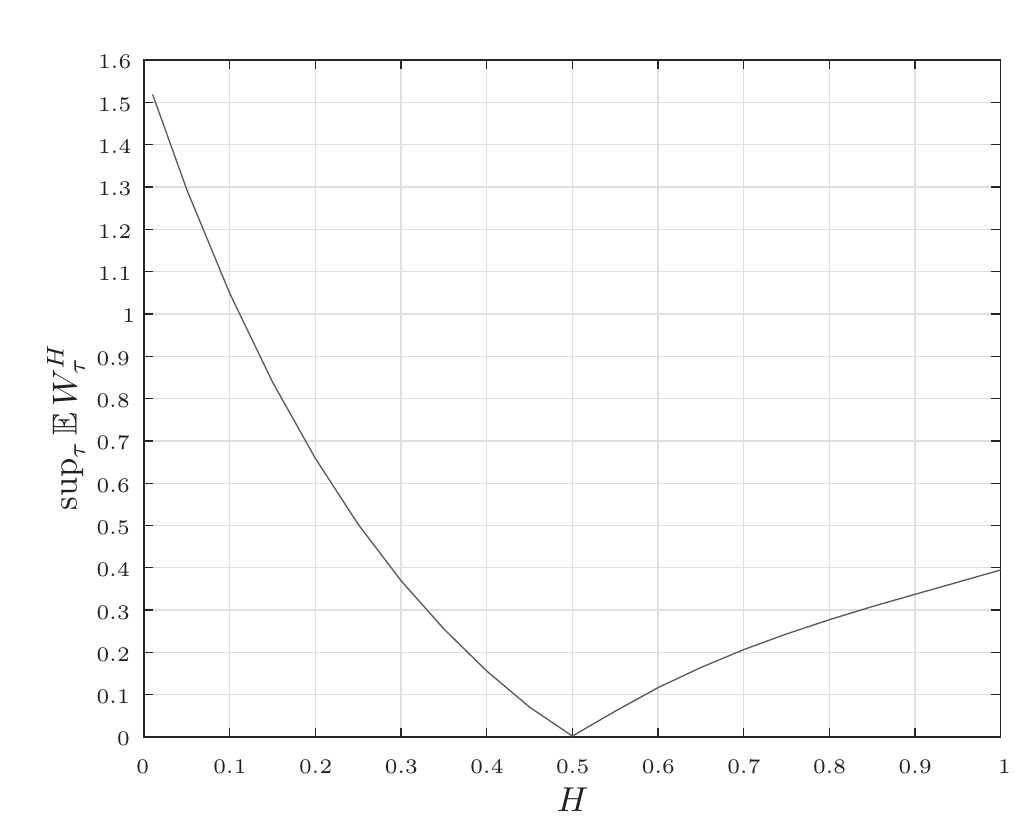}
\caption{\label{fig:fbm}Estimates of $\sup_{\tau \in \crl{0,t_1, \dots, 1}} \mathbb{E}\, W^H_{ \tau }$ for 
different values of $H$.}
\end{figure}

\end{document}